\documentclass{article}
\usepackage{enumerate,amsmath,amssymb,mathabx,amsthm}
\usepackage{algorithm,mathrsfs,graphicx,algorithmicx,authblk}
\usepackage{geometry}
\usepackage[nottoc]{tocbibind}

\newtheorem{theorem}{Theorem}[section]
\newtheorem{lemma}[theorem]{Lemma}
\newtheorem{corollary}[theorem]{Corollary}
\newtheorem{remark}{Remark}[section]

\newtheorem{proposition}{Proposition}[section]
\newtheorem{assumption}{Assumption}
\usepackage[colorlinks,linkcolor=blue,citecolor=red]{hyperref}
\usepackage[noend]{algpseudocode}

\counterwithin{equation}{section}
\counterwithin{figure}{section}
\counterwithin{table}{section}
\title{On the discrete Poincar\'e inequality for B-schemes of 1D Fokker-Planck equations in full space}

\author[a]{Lei Li\thanks{E-mail:leili2010@sjtu.edu.cn}}
\author[b]{Jian-Guo Liu\thanks{E-mail:jliu@math.duke.edu}}
\author[c]{Zhen Wang\thanks{E-mail:wang-zhen@sjtu.edu.cn}}
\affil[a]{School of Mathematical Sciences, Institute of Natural Sciences, MOE-LSC, Shanghai Jiao Tong University, Shanghai, 200240, P.R.China.}
\affil[b]{Department of Mathematics, Department of Physics, Duke University, Durham, NC 27708, USA.}
\affil[c]{School of Mathematical Sciences, Shanghai Jiao Tong University, Shanghai, 200240, P.R.China.}
\date{}

\begin{document}

\maketitle
%\tableofcontents

\begin{abstract}
In this paper, we propose two approaches to derive the discrete Poincar\'e inequality for the B-schemes, a family of finite volume discretization schemes, for the one-dimensional Fokker-Planck equation in full space. We study the properties of the spatially discretized Fokker-Planck equation in the viewpoint of a continuous-time Markov chain. The first approach is based on Gamma-calculus, through which we show that the Bakry-\'Emery criterion still holds in the discrete setting. The second approach employs the Lyapunov function method, allowing us to extend a local discrete Poincar\'e inequality to the full space. The assumptions required for both approaches are roughly comparable with some minor differences. These methods have the potential to be extended to higher dimensions. As a result, we obtain exponential convergence to equilibrium for the discrete schemes by applying the discrete Poincar\'e inequality.
\end{abstract}

\section{Introduction}
The Fokker-Planck equation is a fundamental equation in applied mathematics and theoretical physics, with wide-ranging applications in chemistry, biology, and finance \cite{ceccato2018remarks,karcher2006coarse,furioli2017fokker}. However, obtaining an exact solution is often difficult, which motivates extensive research into numerical simulations.
When designing numerical schemes for the Fokker-Planck equation, it is essential to ensure that discretizations preserve key physical properties, including accurate equilibrium recovery \cite{pareschi2017residual,li2020large}. It is well known that the analysis of long-time behavior for Fokker-Planck-type equations is closely connected to Poincar\'e-type inequalities \cite{arnold2001convex}. In this work, we derive a discrete Poincar\'e inequality in full space for a family of numerical schemes, thereby establishing exponential convergence to equilibrium.

In the continuous setting, Poincar\'e-type inequalities are instrumental in the analysis of partial differential equations. They are commonly used in existence and regularity theories, as well as in studying convergence to equilibrium for kinetic equations. There is also substantial literature on discrete Poincar\'e-type inequalities, which are used to ensure favorable long-time behavior in numerical simulations.
In \cite{bessemoulin2015discrete}, several Poincar\'e-Sobolev inequalities for the Lebesgue measure are extended to the discrete case via continuous embedding into the space of functions of bounded variation. Building on this, \cite{cances2020large} derives a discrete Poincar\'e inequality for the stationary measure $\pi^h$ within the framework of two-point flux approximation and discrete duality finite volume schemes. In \cite{chainais2022long}, discrete Poincar\'e-type inequalities are established for hybrid finite volume schemes and used to demonstrate exponential convergence to discrete steady states.
However, all these results are confined to bounded domains. The discrete Poincar\'e inequality in full space remains largely unexplored and requires further investigation.

It is usually challenging to establish a discrete Poincar\'e inequality in the full space. There is limited literature on this subject
\cite{dujardin2020coercivity,li2020large}, and both works analyze discrete schemes in the viewpoint of Markov chains, a widely adopted approach in the literature. In \cite{delarue2011probabilistic}, the upwind scheme for transport equations is associated with a discrete-time Markov chain, and an $\mathcal{O}(h^{1/2})$ error is obtained due to fluctuations of the chain.
In \cite{dujardin2020coercivity}, a special spatial discretization is proposed for the one-dimensional Fokker-Planck equation with potential $u(x) = x^2/2$, and the scheme is analyzed as a continuous-time Markov chain to derive exponential convergence to equilibrium via a discrete Poincar\'e inequality in the full space.
Our work differs from \cite{dujardin2020coercivity} in that we consider a more general potential and adopt a different discretization of \eqref{FPE}.
In \cite{li2020large}, a one-dimensional discrete Poincar\'e inequality is established through a generalized Hardy inequality, and the long-time behavior of discrete schemes is studied from the Markov chain viewpoint. However, their result requires truncation of coefficients near infinity and thus cannot be directly applied to the particularly important Scharfetter-Gummel scheme \cite{scharfetter1969large}.
In this work, we regard the discrete Fokker-Planck equation as the forward equation of a continuous-time Markov chain and study its long-time behavior using the discrete Poincar\'e inequality.

We focus on the one-dimensional Fokker-Planck equation in this study. Consider the following stochastic differential equation in the It\^o sense \cite{oksendal2003stochastic},
\begin{equation}
    dX_t = -\partial_x u(X_t)\,dt + \sqrt{2}\,dW_t,
\end{equation}
where $u(x)$ is a smooth potential and $W$ is the standard Wiener process defined on a probability space $(\Omega,\mathcal{F},\mathbb{P})$. And $\partial_x u(x) = u'(x)$ in the one-dimensional setting. The well-posedness of this equation is guaranteed under the following drift assumption \cite[Theorem 2.3.5]{mao2007stochastic}:

\begin{assumption}\label{A0}
    The potential $u$ is smooth, and there exist constants $a > 0$, $M > 0$ such that $x\partial_x u(x) \geq a|x|^2$ for all $|x| > M$.
\end{assumption}

Let $\rho_t$ be the probability density of $X_t$ at time $t$. Then $\rho_t$
satisfies the Fokker-Planck equation, also known as the forward equation:
\begin{equation}\label{FPE}
    \partial_t \rho = \partial_x(\rho \partial_xu + \partial_x \rho) =: \mathcal{L}^* \rho,
\end{equation}
Under Assumption \ref{A0}, there exists a unique invariant solution $\pi(x) \propto e^{-u(x)}$ to \eqref{FPE}. Since
\begin{equation}
    \pi \partial_xu + \partial_x\pi = 0,
\end{equation}
the flux in \eqref{FPE} vanishes for the invariant solution, implying that the detailed balance condition holds and the process is reversible \cite{weinan2021applied}. Let $q_t = \rho_t / \pi$, and suppose $\rho_t$ satisfies \eqref{FPE}. Then we obtain the backward equation:
\begin{equation}\label{BWE}
    \partial_t q = -\partial_xu \partial_x q + \partial_{xx} q =: \mathcal{L} q.
\end{equation}
It can be verified that $\mathcal{L}$ is the adjoint operator of $\mathcal{L}^*$ in $L^2(\mathbb{R})$.

In this work we discuss the spatial discretization of \eqref{FPE} by $B$-scheme with step size $h$ and grid points $x_i=ih$ for $i\in\mathbb{Z}$,
\begin{equation}\label{DFPE}
    \frac{d}{dt}\rho_i=\left(\alpha_{i-1}\rho_{i-1}-\beta_i\rho_i\right)-\left(\alpha_i\rho_i-\beta_{i+1}\rho_{i+1}\right)=:\mathcal{L}_h^*\rho_i,
\end{equation}
where the coefficients given by
\begin{equation}\label{C}
    \alpha_i:=\frac{1}{h^2}B\left(\ln{\frac{\pi(x_i)}{\pi(x_{i+1})}}\right),\quad \beta_i:=\frac{1}{h^2}B\left(\ln{\frac{\pi(x_i)}{\pi(x_{i-1})}}\right).
\end{equation}
Set $u_i=u(x_i)$, then $\alpha_i=h^{-2}B(u_{i+1}-u_i)$ and $\beta_i=h^{-2}B(u_{i-1}-u_i)$. Here \eqref{DFPE} can be regarded as the forward equation of a continuous-time Markov chain, where $\alpha_i$ is the transition rate from site $i-1$ to $i$, while $\beta_i$ the rate from site $i+1$ to $i$. Hence this scheme can be simulated by some Monte Carlo methods \cite{rubinstein2016simulation}. 
The conditions we put on $B(s)$ are given as follows.
\begin{assumption}\label{ass:B}
The function $B\in C^1(\mathbb{R})$ and satisfies the following properties:
\begin{itemize}
    \item The function $B(s)$ is Lipschitz continuous;
    \item $B(0)=1$ and $B(s)>0$ for all $s\in\mathbb{R}$;
    \item $\ln{B(-s)}-\ln{B(s)}=s$ for all $s\in\mathbb{R}$;
    \item The function $B(s)$ is monotonically decreasing.
\end{itemize}
\end{assumption}
The first three conditions are relatively standard. Note that the third one also appears in \cite{heida2021consistency}, but is different from the original one \cite{chainais2011finite}, which assumes $B(-s)-B(s)=s$ to simulate $\partial_xu$ in the flux $\rho\partial_xu+\partial_x\rho$. Since $\ln{(1+s)}=s+\mathcal{O}(s^2)$, the modified condition does not lose the consistency, and the benefit is clear. It can be verified the detailed balance condition holds
\begin{equation}\label{DB}
    \alpha_{i-1}\pi(x_{i-1})-\beta_i\pi(x_i)=0,
\end{equation}
then $\{\pi(x_i)\}_{i\in\mathbb{Z}}$ is a stationary solution of \eqref{DFPE} \cite{kelly2011reversibility}. We can define
\begin{equation}\label{pi}
    \pi_i:=\frac{\pi(x_i)}{\sum_{i\in\mathbb{Z}}{\pi(x_i)}},
\end{equation}
and $\pi^h:=\{\pi_i\}_{i\in\mathbb{Z}}$ is the stationary distribution of \eqref{DFPE}.
%\begin{remark}
The conditions $\ln{B(-s)}-\ln{B(s)}=s$ combined with $B(0)=1$ imply that 
\[
B'(0)=-\frac{1}{2}.
\]
%\end{remark}
%In this work, we further assume that $B(s)$ satisfies the following condition:
%\begin{itemize}
%    \item The function $B(s)$ is monotonically decreasing.
%\end{itemize}
The last condition in Assumption \ref{ass:B} is an extra condition we impose. Since the drift and diffusion terms are treated simultaneously as the numerical flux, the monotonicity of $B(s)$ ensures that $\alpha_i$ becomes the primary contributor to the flux near $-\infty$, while $\beta_i$ dominates in the flux near $+\infty$. Consequently, the mass remains concentrated within a bounded domain and does not escape to infinity. Using the detailed balance condition \eqref{DB}, we obtain
\begin{equation*}
    \sum_{i\in\mathbb{Z}} \alpha_i \pi_i = \sum_{i=0}^{+\infty} \alpha_i \pi_i + \sum_{i=0}^{+\infty} \beta_{-i} \pi_{-i},
\end{equation*}
from which the following result follows, due to the boundedness of $\alpha_i$ as $i \to +\infty$ and of $\beta_i$ as $i \to -\infty$.

\begin{lemma}\label{L}
    Under Assumption \ref{A0} and \ref{ass:B}, it holds that
    \begin{equation}
        \sum_{i\in\mathbb{Z}} \left( \alpha_i + \beta_i \right) \pi_i < +\infty.
    \end{equation}
\end{lemma}

\begin{remark}
    The classical Scharfetter-Gummel scheme is recovered by taking $B(s):=s/(e^s-1)$. It can be verified that the function $B(s)$ is monotonically decreasing and convex in this case.
\end{remark}

Define $\langle f^h, g^h \rangle := \sum_{i\in\mathbb{Z}} f_i g_i$, and let $\ell^2$ denote the Hilbert space $\mathbb{R}^\mathbb{Z}$ equipped with the inner product $\langle \cdot, \cdot \rangle$. Furthermore, in this work we primarily consider the following Banach spaces:
\begin{align*}
    &\ell^p := \left\{ f^h \in \mathbb{R}^\mathbb{Z} \ \big| \ \| f^h \|_{\ell^p} := \left( \sum\nolimits_{i\in\mathbb{Z}} |f_i|^p \right)^{1/p} < +\infty \right\}, \quad 1 \leq p < +\infty,\\
    &\ell^p(\pi^h) := \left\{ f^h \in \mathbb{R}^\mathbb{Z} \ \big| \ \| f^h \|_{\ell^p(\pi^h)} := \left( \sum\nolimits_{i\in\mathbb{Z}} |f_i|^p \pi_i \right)^{1/p} < +\infty \right\}, \quad 1 \leq p < +\infty,\\
    &\ell^\infty = \ell^\infty(\pi^h) := \left\{ f^h \in \mathbb{R}^\mathbb{Z} \ \big| \ \| f^h \|_{\ell^\infty} := \sup\nolimits_{i\in\mathbb{Z}} |f_i| < +\infty \right\}.
\end{align*}
Let $q_i(t) := \rho_i(t)/\pi_i$, where $\rho_i(t)$ satisfies \eqref{DFPE}. Since the detailed balance condition \eqref{DB} holds, $q_i(t)$ satisfies the discrete backward equation
\begin{equation}\label{DBWE}
    \frac{d}{dt} q_i = \alpha_i (q_{i+1} - q_i) - \beta_i (q_i - q_{i-1}) =: \mathcal{L}_h q_i.
\end{equation}
It can be verified that $\mathcal{L}_h$ is the adjoint operator of $\mathcal{L}^*_h$ in $\ell^2$. Moreover, the operator $\mathcal{L}_h$ is symmetric with respect to $\pi^h$ in the sense that
\begin{equation*}
    \left\langle g^h \mathcal{L}_h f^h, \pi^h \right\rangle = \left\langle f^h \mathcal{L}_h g^h, \pi^h \right\rangle, \quad \forall f^h, g^h \in \ell^2(\pi^h).
\end{equation*}

In \cite{heida2021consistency}, the consistency of modified $B$-schemes is studied from the perspective of weighted Stolarsky means \cite{stolarsky1975generalizations}, and convergence of order $\mathcal{O}(h^2)$ is obtained in the energy norm for a bounded domain with uniform grids. In this work, we are concerned with the long-time behavior of the discrete scheme in the full space. Since, by \eqref{pi}, the stationary solution of the discrete equation \eqref{DFPE} coincides with that of the continuous equation \eqref{FPE}, we can establish stability and convergence of the long-time behavior once the discrete Poincar\'e inequality and the ergodicity of \eqref{DFPE} are proved.

We define the following difference operators:
\begin{equation*}
    \Delta^+ f_i := f_{i+1} - f_i, \qquad \Delta^- f_i := f_i - f_{i-1},
\end{equation*}
and
\begin{equation*}
    \left( \left| Df^h \right| \right)_i := \left( \frac{1}{2} \alpha_i (\Delta^+ f_i)^2 + \frac{1}{2} \beta_i (\Delta^- f_i)^2 \right)^{\frac{1}{2}}.
\end{equation*}
Our goal in this work is to establish discrete Poincar\'e inequalities in the full space for the stationary distribution $\pi^h$ given in \eqref{pi}, in the sense that
\begin{equation}\label{PI}
    \sum_{i \in \mathbb{Z}} \left( f_i - \sum_{j \in \mathbb{Z}} f_j \pi_j \right)^2 \pi_i \leq \frac{1}{\kappa} \left\| Df^h \right\|_{\ell^2(\pi^h)},
\end{equation}
where $\kappa$ is a positive constant independent of $h$. We mainly consider $f^h$ in the following function space:
\begin{equation}
    f^h \in \mathcal{H}^1(\pi^h) := \left\{ g^h \in \mathbb{R}^{\mathbb{Z}} \ \big| \ \left\| g^h \right\|^2_{\mathcal{H}^1(\pi^h)} := \left\| g^h \right\|^2_{\ell^2(\pi^h)} + \left\| Dg^h \right\|^2_{\ell^2(\pi^h)} < +\infty \right\}.
\end{equation}

We define the \textit{carr\'e du champ} operator \cite{baudoin2017bakry} as
\begin{equation*}
    \Gamma^h(f_i, g_i) := \frac{1}{2} \left( \mathcal{L}_h(f_i g_i) - f_i \mathcal{L}_h g_i - g_i \mathcal{L}_h f_i \right),
\end{equation*}
and it can be verified that
\begin{equation}\label{G}
    \Gamma^h(f_i, g_i) = \frac{1}{2} \alpha_i \Delta^+ f_i \Delta^+ g_i + \frac{1}{2} \beta_i \Delta^- f_i \Delta^- g_i.
\end{equation}
Then the inequality \eqref{PI} can be equivalently rewritten in the following form
\begin{equation}
    \left\langle (f^h)^2, \pi^h \right\rangle - \left\langle f^h, \pi^h \right\rangle^2 \leq \frac{1}{\kappa} \left\langle \Gamma^h(f^h, f^h), \pi^h \right\rangle.
\end{equation}

\begin{remark}
   By Lemma \ref{L}, under Assumption \ref{A0} and \ref{ass:B}, we have $\ell^\infty \subset \mathcal{H}^1(\pi^h)$. Since $\ell^\infty$ is dense in $\mathcal{H}^1(\pi^h)$, the discrete Poincar'e inequality for $f \in \ell^\infty$ extends to all $f \in \mathcal{H}^1(\pi^h)$.
\end{remark}

Let $P_t$ be the semigroup generated by $\mathcal{L}_h$, and $P^*_t$ its dual. Viewing \eqref{DFPE} from the perspective of a jump process, in Section \ref{S1} we will recall some properties of continuous-time Markov chains and use them to establish the well-posedness of $P_t: \ell^\infty \to \ell^\infty$, and consequently for its dual $P^*_t: \ell^1 \to \ell^1$. After showing that $\mathcal{L}_h^*$ is the generator of $P_t^*$, we can establish the well-posedness of both \eqref{DFPE} and \eqref{DBWE}. In Section \ref{S1}, we will also show that $\|P_t f^h\|_{\ell^\infty} \leq \|f^h\|_{\ell^\infty}$ for all $t > 0$. Then, together with Lemma \ref{L} and the dominated convergence theorem, we obtain the following result.

\begin{proposition}\label{T}
    Under Assumption \ref{A0} and \ref{ass:B}, let $\varphi: \mathbb{R} \to \mathbb{R}$ be a smooth convex function. Then the following identity holds for all $f^h \in \ell^\infty$,
    \begin{equation}
        \frac{d}{dt} \left\langle \varphi(P_t f^h), \pi^h \right\rangle
        = \left\langle \varphi'\left(P_t f^h\right) \mathcal{L}_h P_t f^h, \pi^h \right\rangle
        = -\frac{1}{2} \left\langle \Gamma^h\left( \varphi'\left(P_t f^h\right), P_t f^h \right), \pi^h \right\rangle.
    \end{equation}
\end{proposition}

We omit the proof of Proposition \ref{T} since similar arguments are used to derive the expressions for both $P_t f^h$ and $P_t^* g^h$ in Section \ref{S1}. The key point lies in the boundedness of $(\varphi'(P_t f^h) \mathcal{L}_h P_t f^h)_i$ by $\alpha_i + \beta_i$, uniformly in $t$.

\begin{remark}
    Taking $\varphi(x) = x^2$, we see that the discrete Poincaré inequality implies exponential convergence in $\ell^2(\pi^h)$ via Grönwall's inequality. Moreover, the convexity of $\varphi$ guarantees $\Gamma^h(\varphi'(f^h), f^h) \geq 0$, which ensures that $\langle \varphi(P_t f^h), \pi^h \rangle$ is non-increasing in time.
\end{remark}

In the continuous setting, the Poincar\'e inequality holds for the stationary distribution $\pi$ if $\partial_{xx}u \geq \lambda > 0$, a condition also known as the Bakry-\'Emery criterion \cite{bakry2006diffusions}, which clearly implies Assumption \ref{A0}. We will show that this criterion still holds in the discrete setting. Here, $\partial_{xx} u = u''$ in the one-dimensional case. In \cite{baudoin2017bakry}, Baudoin computes second-order derivatives and uses the \textit{carr\'e du champ} operator to establish exponential convergence to equilibrium for the kinetic Fokker-Planck equation. Inspired by this work, we introduce a discrete \textit{carr\'e du champ} operator and derive the discrete Poincar\'e inequality in Section \ref{S2}. Note that the strong convexity condition $\partial_{xx}u \geq \lambda > 0$ ensures Assumption \ref{A0}.

\begin{theorem}\label{T1}
    Suppose that $B(s)$ satisfies the properties in Assumption \ref{ass:B} and is also convex. If the potential satisfies $\partial_{xx}u \geq \lambda$ for some constant $\lambda > 0$, then the discrete Poincar\'e inequality holds for $\pi^h$, and there exists a constant $\kappa > 0$ independent of $h$, such that the following inequality holds for all $f^h \in \mathcal{H}^1(\pi^h)$,
    \begin{equation}\label{GDPI}
        \left\langle (f^h)^2, \pi^h \right\rangle - \left\langle f^h, \pi^h \right\rangle^2
        \leq \frac{1}{\kappa} \left\langle \Gamma^h(f^h, f^h), \pi^h \right\rangle.
    \end{equation}
\end{theorem}

\begin{remark}
    For any $f^h \in \ell^2(\pi^h) \setminus \mathcal{H}^1(\pi^h)$, the right-hand side of \eqref{GDPI} becomes infinite, making the inequality trivially satisfied. Therefore, our result establishes the discrete Poincar\'e inequality for all $f^h \in \ell^2(\pi^h)$.
\end{remark}

In the continuous setting, it is known that the strong convexity condition $\partial_{xx} u \geq \lambda > 0$ is sufficiently strong to imply the exponential convergence in $H^1(\pi)$. Under the same convexity assumption, we can establish analogous exponential convergence in $\mathcal{H}^1(\pi^h)$.

\begin{corollary}\label{C1}
    Suppose the conditions in Theorem \ref{T1} hold. Then there exists a constant $\theta > 0$ independent of $h$ such that the following inequality holds for all $f^h \in \ell^\infty$,
    \begin{equation}
        \begin{aligned}
            &\left\langle \left(P_t f^h - \langle P_t f^h, \pi^h \rangle \right)^2 + \Gamma^h(P_t f^h, P_t f^h), \pi^h \right\rangle \\
            &\qquad \leq e^{-\theta t} \left\langle \left(f^h - \langle f^h, \pi^h \rangle \right)^2 + \Gamma^h(f^h, f^h), \pi^h \right\rangle.
        \end{aligned}
    \end{equation}
\end{corollary}

As in the continuous case, we can generalize the above result via a Holley-Stroock perturbation-type argument \cite{holley1986logarithmic}. When $u$ differs from $\widetilde{u}$ on a compact set and $|u - \widetilde{u}|$ is bounded, the corresponding stationary distributions are equivalent. This yields the following result.

\begin{corollary}\label{C2}
    Suppose that $B(s)$ satisfies the properties in Assumption \ref{ass:B} and is also convex. If the potential $u$ admits a decomposition $u=\widetilde{u}+v$, where $\widetilde{u}$ is strongly convex and $v$ is a bounded perturbation
    supported on some compact set, then the discrete Poincar\'e inequality holds for $\pi^h$, and there exists a constant $\kappa > 0$ independent of $h$ such that inequality \eqref{GDPI} holds for all $f^h \in \mathcal{H}^1(\pi^h)$.
\end{corollary}

In \cite{bakry2008rate}, Bakry, Cattiaux, and Guillin have employed the powerful method of Lyapunov functions to prove the Poincar\'e inequality for a log-concave stationary distribution. Following their observation that $\langle \Gamma(f^2 / W, W), \pi \rangle \leq \langle \Gamma(f, f), \pi \rangle$ holds for all smooth functions $f$ and $W$, we will derive an analogous discrete result and establish the discrete Poincar\'e inequality in the full space in Section \ref{S3}. Notably, we do not require convexity of $B$ or $u$ in the following result.

\begin{theorem}\label{T2}
    Under Assumption \ref{A0} and \ref{ass:B}, there exists $h_0 > 0$ such that for any $0 < h < h_0$, the discrete Poincar\'e inequality holds for $\pi^h$,  and there exists a constant $\kappa > 0$ independent of $h$ such that inequality \eqref{GDPI} holds for all $f^h \in \mathcal{H}^1(\pi^h)$.
\end{theorem}

\begin{remark}
    The difference in scope between Corollary \ref{C2} (valid for all $h > 0$) and Theorem \ref{T2} (requiring $h < h_0$) stems from their distinct analytical frameworks. Corollary \ref{C2} follows from $\Gamma$-calculus applied to the backward equation of a continuous-time Markov chain, which inherently preserves ergodicity for all $h > 0$. The proof of Theorem \ref{T2} relies directly on the discrete backward equation \eqref{DBWE}, and the restriction to small $h$ arises from the consistency in the discrete approximation during the construction of Lyapunov function.
\end{remark}

In this work, we regard the discrete equation \eqref{DFPE} as the forward equation of a jump process. In Section \ref{S1}, we establish the well-posedness and ergodicity of the associated semigroup $P_t$ by leveraging fundamental properties of continuous-time Markov chains. Building on this foundation, in Section \ref{S2}, we discuss the operator $\Gamma^h$ and its iterated form $\Gamma_2^h$, derive a lower bound for $\Gamma^h_2$ in terms of $\Gamma^h$, and and present the proof of Theorem \ref{T1} alongside its corollaries. Finally, in Section \ref{S3}, we adapt techniques from the continuous setting to construct a Lyapunov function for the discrete problem, which enables the extension of a local discrete Poincar\'e inequality to the full space, culminating in the proof of Theorem \ref{T2}.

\section{Properties of the continuous-time Markov chain}\label{S1}
In this section, we introduce some fundamental properties of continuous-time Markov chains, also known as jump processes when defined on a discrete state space. Let $Z_t$ be a jump process corresponding to \eqref{DFPE}, and define $p_t(i, j) := \mathbb{P}(Z_t = j\ |\ Z_0 = i) \geq 0$. Then $p_t(i, j)$ satisfies the forward equation
\begin{equation}\label{GFWE}
    \frac{d}{dt} p_t(i, j) = \alpha_{j-1} p_t(i, j-1) + \beta_{j+1} p_t(i, j+1) - (\alpha_j + \beta_j) p_t(i, j),
\end{equation}
as well as the backward equation \cite[Theorem 2.14]{liggett2010continuous}
\begin{equation}\label{GBWE}
    \frac{d}{dt} p_t(i, j) = \alpha_i p_t(i+1, j) + \beta_i p_t(i-1, j) - (\alpha_i + \beta_i) p_t(i, j).
\end{equation}
Define the $Q$-matrix by $Q(i, j) := \lim_{t \to 0} p_t(i, j)/t$. Then
\begin{equation*}
    Q(i, i+1) = \alpha_i, \quad Q(i, i) = -(\alpha_i + \beta_i), \quad Q(i, i-1) = \beta_i.
\end{equation*}
Let $\tau_i := \inf_{t > 0} \{t\ |\ Z_0 = i,\ Z_t \neq i\}$ denote the first exit time from state $i$. By \cite[Proposition 2.29]{liggett2010continuous},
\begin{equation*}
    \mathbb{P}(\tau_i > t) = e^{-(\alpha_i + \beta_i)t},
\end{equation*}
and the transition probabilities at the first jump are
\begin{equation*}
    \mathbb{P}(Z_{\tau_i} = i+1\ |\ Z_0 = i) = \frac{\alpha_i}{\alpha_i + \beta_i}, \quad
    \mathbb{P}(Z_{\tau_i} = i-1\ |\ Z_0 = i) = \frac{\beta_i}{\alpha_i + \beta_i}.
\end{equation*}

A Markov chain is irreducible if $p_t(i,j)>0$ for all $i,j\in\mathbb{Z}$ and $t>0$ \cite[Definition 2.47]{liggett2010continuous}. Since $Q(i,i-1)$ and $Q(i,i+1)$ are positive for all $i\in\mathbb{Z}$, we can claim that $Z_t$ is an irreducible Markov chain. Recall that the detailed balance condition \eqref{DB} holds, and consequently $\pi^h$ is the unique stationary distribution of $Z_t$. Then the transition probabilities converge to the stationary distribution by \cite[Theorem 2.66]{liggett2010continuous},
\begin{equation}
    \lim_{t\to+\infty}{p_t(i,j)}=\pi_j,\quad\forall\ i,j\in\mathbb{Z}.
\end{equation}
Since \cite[Corollary 2.34]{liggett2010continuous} ensures that the total probability is conserved $\sum_{j\in\mathbb{Z}}{p_t(i,j)}=1$, then by \cite[Theorem 2.26]{liggett2010continuous}, $p_t(i,j)$ is the unique solution to \eqref{GBWE} with initial state $p_0(i,j)=\delta_{ij}$. Consequently, $p_t(i,j)=0$ is the unique solution corresponding to the initial condition $p_0(i,j)=0$, leading to the uniqueness of solutions to the backward equation \eqref{GBWE}.

Next we consider a general initial state $f^h\in\ell^\infty$, that is, the sequence $f^h=\{f_i\}_{i\in\mathbb{Z}}$ is bounded. Let $P_t$ be the semi-group generated by $\mathcal{L}_h$. Define $f^h_t$ as
\begin{equation*}
    \left(f^h_t\right)_i:=\mathbb{E}\left[f^h(Z_t)\ |\ Z_0=i\right]=\sum_{j\in\mathbb{Z}}{f_jp_t(i,j)},
\end{equation*}
then one can get the maximum principle
\begin{equation}
    \left|\left(f^h_t\right)_i\right|=\left|\mathbb{E}\left[f^h(Z_t)\ |\ Z_0=i\right]\right|\leq\mathbb{E}\left[\left\|f^h(Z_t)\right\|_{\ell^\infty}\ |\ Z_0=i\right]=\left\|f^h\right\|_{\ell^\infty},
\end{equation}
that is, $P_tf^h\in\ell^\infty$ for all $f^h\in\ell^\infty$. For a fixed index $i$, it follows by Fubini's theorem and the backward equation \eqref{GBWE} that
\begin{align*}
    &\left(f^h_{t+\Delta t}\right)_i-\left(f^h_t\right)_i\\
    &\quad=\sum_{j\in\mathbb{Z}}{f_j\Delta t\int^1_0}{p'_{t+s\Delta t}(i,j)}ds\\
    &\quad=\sum_{j\in\mathbb{Z}}{f_j\Delta t\int^1_0{\left[\alpha_ip_{t+s\Delta t}(i+1,j)+\beta_ip_{t+s\Delta t}(i-1,j)-(\alpha_i+\beta_i)p_{t+s\Delta t}(i,j)\right]}ds}\\
    &\quad=\int^1_0{\sum_{j\in\mathbb{Z}}{f_j\Delta t\left[\alpha_ip_{t+s\Delta t}(i+1,j)+\beta_ip_{t+s\Delta t}(i-1,j)-(\alpha_i+\beta_i)p_{t+s\Delta t}(i,j)\right]}}ds,
\end{align*}
and by dominated convergence theorem,
\begin{align*}
    \frac{d}{dt}\left(f^h_t\right)_i
    &=\sum_{j\in\mathbb{Z}}{f_j\left[\alpha_ip_t(i+1,j)+\beta_ip_t(i-1,j)-(\alpha_i+\beta_i)p_t(i,j)\right]}\\
    &=\alpha_i\left(\left(f^h_t\right)_{i+1}-\left(f^h_t\right)_i\right)-\beta_i\left(\left(f^h_t\right)_i-\left(f^h_t\right)_{i-1}\right),
\end{align*}
that is, $f^h_t$ satisfies the backward equation \eqref{DBWE}. Then we claim that $P_tf^h=f^h_t$ by the uniqueness of solutions.

For a fixed index $i$, since $p_t(i,j)$ satisfies the forward equation \eqref{DFPE} with respect to $j$, then $p_t(i,j)/\pi_j$ will satisfy the backward equation \eqref{DBWE}. By the maximum principle,
\begin{equation*}
    \frac{p_t(i,j)}{\pi_j}\leq\max_{j\in\mathbb{Z}}{\frac{p_0(i,j)}{\pi_j}}=\frac{1}{\pi_i},
\end{equation*}
that is, $f_jp_t(i,j)/\pi_j$ is uniformly bounded with respect to $j$. It follows by dominated convergence theorem that
\begin{equation*}
    \lim_{t\to+\infty}{\left(P_tf^h\right)_i}=\lim_{t\to+\infty}{\sum_{j\in\mathbb{Z}}{f_jp_t(i,j)}}=\lim_{t\to+\infty}{\sum_{j\in\mathbb{Z}}{f_j\frac{p_t(i,j)}{\pi_j}\pi_j}}=\sum_{j\in\mathbb{Z}}{f_j\pi_j}=\left\langle f^h,\pi^h\right\rangle,
\end{equation*}
and by the uniform boundedness of $P_tf^h$,
\begin{equation*}
    \lim_{t\to+\infty}{\left\langle\left(P_tf^h-\left\langle f^h,\pi^h\right\rangle\right)^2,\pi^h\right\rangle}=0.
\end{equation*}

Now we have obtained the well-posedness for the semi-group $P_t:\ell^\infty\to\ell^\infty$, and thus for its dual semi-group $P^*_t:\ell^1\to\ell^1$ defined by
\begin{equation}
    \left\langle f^h,P^*_tg^h\right\rangle=\left\langle P_tf^h,g^h\right\rangle,\qquad\forall\ f^h\in\ell^\infty,\ g^h\in\ell^1.
\end{equation}
And it can be verified that
\begin{align*}
    \left\langle f^h,P^*_tg^h\right\rangle
    &=\sum_{i\in\mathbb{Z}}{(P_tf^h)_ig_i}\\
    &=\sum_{i\in\mathbb{Z}}{\left(\sum_{j\in\mathbb{Z}}{f_jp_t(i,j)}\right)g_i}\\
    &=\sum_{j\in\mathbb{Z}}{f_j\left(\sum_{i\in\mathbb{Z}}{g_ip_t(i,j)}\right)}.
\end{align*}
Since $0\leq p_t(i,j)\leq1$ for all $i,j\in\mathbb{Z}$ and $t>0$, one can have
\begin{equation}
    \left(P^*_tg^h\right)_j=\sum_{i\in\mathbb{Z}}{g_ip_t(i,j)}.
\end{equation}
It remains to show $P^*_tg^h$ satisfies \eqref{DFPE}. Let $\Delta t>0$ and fix an index $j$, it follows by Fubini's theorem and the forward equation \eqref{GFWE} that
\begin{align*}
    &\left(P^*_{t+\Delta t}g^h\right)_j-\left(P^*_tg^h\right)_j\\
    &\qquad=\sum_{i\in\mathbb{Z}}{g_i\left(p_{t+\Delta t}(i,j)-p_t(i,j)\right)}\\
    &\qquad=\sum_{i\in\mathbb{Z}}{g_i\Delta t\int^1_0{\left[\alpha_{j-1}p_{t+s\Delta t}(i,j-1)+\beta_{j+1}p_{t+s\Delta t}(i,j+1)-\left(\alpha_j+\beta_j\right)p_{t+s\Delta t}(i,j)\right]}ds}\\
    &\qquad=\int^1_0{\sum_{i\in\mathbb{Z}}{g_i\Delta t\left[\alpha_{j-1}p_{t+s\Delta t}(i,j-1)+\beta_{j+1}p_{t+s\Delta t}(i,j+1)-\left(\alpha_j+\beta_j\right)p_{t+s\Delta t}(i,j)\right]}}ds,
\end{align*}
and by dominated convergence theorem one can have
\begin{align*}
    \frac{d}{dt}\left(P^*_tg^h\right)_j
    &=\sum_{i\in\mathbb{Z}}{g_i\left[\alpha_{j-1}p_t(i,j-1)+\beta_{j+1}p_t(i,j+1)-\left(\alpha_j+\beta_j\right)p_t(i,j)\right]}\\
    &=\alpha_{j-1}\left(P^*_tg^h\right)_{j-1}+\beta_{j+1}\left(P^*_tg^h\right)_{j+1}-\left(\alpha_j+\beta_j\right)\left(P^*g^h\right)_j,
\end{align*}
that is, $\mathcal{L}^*_h$ is the generator of the semi-group $P^*_t$.

We summarize the above results as follows:
\begin{proposition}\label{ptf}
    Under Assumption \ref{A0} and \ref{ass:B}, the semi-group $P^*_t:\ell^1\to\ell^1$ generated by $\mathcal{L}^*_h$ is well-posed,
    \begin{equation*}
        \left(P^*_tg^h\right)_j=\sum_{i\in\mathbb{Z}}{g_ip_t(i,j)},\qquad\forall\ g^h\in\ell^1,
    \end{equation*}
    and the semi-group $P_t:\ell^\infty\to\ell^\infty$ generated by $\mathcal{L}_h$ is well-posed,
    \begin{equation*}
        \left(P_tf^h\right)_i=\sum_{j\in\mathbb{Z}}{f_jp_t(i,j)},\qquad\forall\ f^h\in\ell^\infty.
    \end{equation*}
    Furthermore, the semi-group $P_t$ enjoys the following properties:
    \begin{itemize}
        \item maximum principle: $\|P_tf^h\|_{\ell^\infty}\leq\|f^h\|_{\ell^\infty}$;
        \item ergodicity: $\lim_{t\to+\infty}P_tf^h=\langle f^h,\pi^h\rangle$;
        \item $\ell^2(\pi^h)$-convergence: $\lim_{t\to+\infty}{\langle(P_tf^h-\langle f^h,\pi^h\rangle)^2,\pi^h\rangle}=0$.
    \end{itemize}
\end{proposition}

\section{Discrete Poincar\'e inequality by $\Gamma$-calculus}\label{S2}
In this section, we introduce our first approach to establish the discrete Poincar\'e inequality through $\Gamma$-Calculus and \textit{carr\'e du champ} operator \eqref{G}. Define the iteration of the \textit{carr\'e du champ} operator as
\begin{equation*}
    \Gamma^h_2(f_i,g_i):=\frac{1}{2}\left(\mathcal{L}_h\Gamma^h(f_i,g_i)-\Gamma^h(f_i,\mathcal{L}_hg_i)-\Gamma^h(\mathcal{L}_hf_i,g_i)\right).
\end{equation*}
Following the idea in \cite{baudoin2017bakry}, one can verify the following properties by definition.

\begin{proposition}\label{P1}
    Let $P_t$ be the semi-group generated by $\mathcal{L}_h$. Then for any $f^h\in\ell^2(\pi^h)$ and $0<t<T$,
    \begin{equation*}
        \frac{d}{dt}P_t\left[(P_{T-t}f^h)^2\right]=2P_t\left[\Gamma^h(P_{T-t}f^h,P_{T-t}f^h)\right],
    \end{equation*}
    and
    \begin{equation*}
        \frac{d}{dt}P_t\left[\Gamma^h(P_{T-t}f^h,P_{T-t}f^h)\right]=2P_t\left[\Gamma^h_2(P_{T-t}f^h,P_{T-t}f^h)\right].
    \end{equation*}
\end{proposition}

For simplicity, we adopt the componentwise notation where $f^h\geq g^h$ means $f_i\geq g_i$ for all $i\in\mathbb{N}$, and $f^hg^h$ denotes the pointwise product $(f^hg^h)_i=f_ig_i$. Then we can introduce the curvature condition in Lemma \ref{DCC}, and we give the proof in subsection \ref{S2.2}. In the continuous setting, by Bakry-\'Emery criterion we denote $\Gamma_2(f,f)\geq\lambda\Gamma(f,f)$ for some constant $\lambda>0$, with the largest $\lambda$ called the Bakry-\'Emery curvature.

\begin{lemma}\label{DCC}
    Suppose that $B(s)$ satisfies the properties in Assumption \ref{ass:B} and is also convex. If the potential satisfies $\partial_{xx}u\geq\lambda$ for some constant $\lambda>0$, then there exists a constant $\widetilde{\lambda}>0$ such that the following inequality holds for all $f^h\in\ell^2(\pi^h)$,
    \begin{equation}\label{CC}
        \Gamma^h_2(f^h,f^h)\geq\widetilde{\lambda}\Gamma^h(f^h,f^h).
    \end{equation}
\end{lemma}

Combining Proposition \ref{P1} with Lemma \ref{DCC}, by Gr\"ownwall's inequality we can obtain an exponential convergence to zero for the time derivative of $P_t[(P_{T-t}f^h)^2]$, which immediately implies the exponential convergence of $P_t[(P_{T-t}f^h)^2]$ itself. We leave the details in subsection \ref{S2.1} and give the following result
\begin{equation}\label{PI1}
    \left\langle(f^h)^2,\pi^h\right\rangle-\left\langle f^h,\pi^h\right\rangle^2\leq\frac{1}{\widetilde{\lambda}}\left\langle\Gamma^h(f^h,f^h),\pi^h\right\rangle.
\end{equation}
where $f^h\in\mathcal{H}^1(\pi^h)$. And we have established Theorem \ref{T1}.

\begin{remark}
    The discrete Poincar\'e inequality \eqref{PI1} and the curvature condition \eqref{CC} share the same constant $\widetilde{\lambda}$, which is consistent with continuous cases.
\end{remark}

\begin{remark}
    The analysis in \cite{li2020large} requires the technical assumption $\sup_{s\in\mathbb{R}} B(s) < +\infty$, leading to a modification on the Scharfetter-Gummel scheme to fit their theoretical framework. Our result removes the boundedness constraint on $B(s)$, and directly applies to the original Scharfetter-Gummel scheme while preserving the convergence properties.
\end{remark}

\subsection{Proof of Theorem \ref{T1}}\label{S2.1}
In this subsection, we derive Theorem \ref{T1} through Proposition \ref{P1} and Lemma \ref{DCC}, by which one can have
\begin{equation*}
    \frac{d}{dt}P_t\left[\Gamma^h(P_{T-t}f^h,P_{T-t}f^h)\right]=P_t\left[\Gamma^h_2(P_{T-t}f^h,P_{T-t}f^h)\right]\geq2\widetilde{\lambda}P_t\left[\Gamma^h(P_{T-t}f^h,P_{T-t}f^h)\right].
\end{equation*}
By Gr\"onwall's inequality,
\begin{equation}\label{CoG}
    \Gamma^h(P_Tf^h,P_Tf^h)\leq e^{-2\widetilde{\lambda}T}P_T\left[\Gamma^h(f^h,f^h)\right].
\end{equation}
Since the above inequality holds for all $T>0$, by the monotonicity of semi-group $P_t$ one can obtain
\begin{equation*}
    \frac{d}{dt}P_t\left[(P_{T-t}f^h)^2\right]=2P_t\left[\Gamma^h(P_{T-t}f^h,P_{T-t}f^h)\right]\leq2P_t\left[e^{-2\widetilde{\lambda}(T-t)}P_{T-t}\Gamma^h(f^h,f^h)\right].
\end{equation*}
Note that $P_t[(P_{T-t}f^h)^2]|_{t=0}=(P_Tf^h)^2$ and $P_t[(P_{T-t}f^h)^2]|_{t=T}=P_T[(f^h)^2]$, by the integral inequality one has
\begin{equation*}
    P_T\left[(f^h)^2\right]-\left(P_Tf^h\right)^2\leq2\int^T_0{P_t\left[e^{-2\widetilde{\lambda}(T-t)}P_{T-t}\Gamma^h(f^h,f^h)\right]}dt.
\end{equation*}
Recall that $\pi^h$ is the invariant distribution with respect to $P_t$. Then
\begin{equation*}
    \left\langle P_T\left[(f^h)^2\right]-\left(P_Tf^h\right)^2,\pi^h\right\rangle=\left\langle(f^h)^2,\pi^h\right\rangle-\left\langle(P_Tf^h)^2,\pi^h\right\rangle,
\end{equation*}
and by Fubini's theorem,
\begin{align*}
    \left\langle2\int^T_0{P_t\left[e^{-2\widetilde{\lambda}(T-t)}P_{T-t}\Gamma^h(f^h,f^h)\right]}dt,\pi^h\right\rangle
    &=\int^T_0{\left\langle 2e^{-2\widetilde{\lambda}(T-t)}P_{T-t}\Gamma^h(f^h,f^h),\pi^h\right\rangle}dt\\
    &=\int^T_0{2e^{-2\widetilde{\lambda}(T-t)}\left\langle \Gamma^h(f^h,f^h),\pi^h\right\rangle}dt\\
    &=\frac{1}{\widetilde{\lambda}}\left(1-e^{-2\widetilde{\lambda}T}\right)\left\langle\Gamma^h(f^h,f^h),\pi^h\right\rangle.
\end{align*}
Hence, for arbitrary $T>0$ it holds that
\begin{equation}
    \left\langle(f^h)^2,\pi^h\right\rangle-\left\langle(P_Tf^h)^2,\pi^h\right\rangle\leq\frac{1}{\widetilde{\lambda}}\left(1-e^{-2\widetilde{\lambda}T}\right)\left\langle\Gamma^h(f^h,f^h),\pi^h\right\rangle.
\end{equation}
Let $T\to+\infty$ and one can obtain
\begin{equation*}
    \left\langle(f^h)^2,\pi^h\right\rangle-\lim_{T\to+\infty}{\left\langle(P_Tf^h)^2,\pi^h\right\rangle}\leq\frac{1}{\widetilde{\lambda}}\left\langle\Gamma^h(f^h,f^h),\pi^h\right\rangle.
\end{equation*}

Since $\ell^\infty$ is dense in $f^h\in\mathcal{H}^1(\pi^h)$, the discrete Pioncar\'e inequality for $f^h\in\mathcal{H}^1(\pi^h)$ follows if we can establish the inequality \eqref{GDPI} for all $f^h\in\ell^\infty$. By Proposition \ref{ptf} one can have
\begin{equation*}
    \lim_{T\to+\infty}{\left\langle\left(P_Tf^h\right)^2,\pi^h\right\rangle}=\left\langle f^h,\pi^h\right\rangle^2,
\end{equation*}
that is, for all $f^h\in\ell^\infty$ it holds that
\begin{equation*}
    \left\langle(f^h)^2,\pi^h\right\rangle-\left\langle f^h,\pi^h\right\rangle^2\leq\frac{1}{\widetilde{\lambda}}\left\langle\Gamma^h(f^h,f^h),\pi^h\right\rangle.
\end{equation*}

When $f^h\in\mathcal{H}^1(\pi^h)$, one can set $(f^h_k)_i:=f_i$ for $|i|\leq k$ and $(f^h_k)_i=0$ for $|i|>k$. Then $f^h_k\in\ell^\infty$ for all $k\in\mathbb{N}^+$, and thus $f^h_k$ satisfies the Poincar\'e inequality \eqref{GDPI}. Since $f^h\in\mathcal{H}^1(\pi^h)$, it holds that
\begin{equation*}
    \lim_{k\to+\infty}{\left\langle(f^h_k)^2,\pi^h\right\rangle}=\left\langle(f^h)^2,\pi^h\right\rangle-\lim_{k\to+\infty}{\sum_{|i|>k}{f^2_i\pi_i}}=\left\langle(f^h)^2,\pi^h\right\rangle,
\end{equation*}
as well as
\begin{align*}
    \lim_{k\to+\infty}{\left\langle\Gamma^h(f^h_k,f^h_k),\pi^h\right\rangle}
    &=\left\langle\Gamma^h(f^h,f^h),\pi^h\right\rangle-\lim_{k\to+\infty}{\frac{1}{2}\sum_{|i|>k}{\left(\alpha_i(\Delta^+f_i)^2+\beta_i(\Delta^-f_i)^2\right)\pi_i}}\\
    &=\left\langle\Gamma^h(f^h,f^h),\pi^h\right\rangle.
\end{align*}
By H\"older's inequality,
\begin{equation*}
    \lim_{k\to+\infty}\left|{\sum_{|i|>k}{f_i\pi_i}}\right|\leq\lim_{k\to+\infty}{\left(\sum_{|i|>k}{f^2_i\pi_i}\right)^{\frac{1}{2}}\left(\sum_{|i|>l}1^2\pi_i\right)^{\frac{1}{2}}}=0,
\end{equation*}
and thus
\begin{equation*}
    \lim_{k\to+\infty}{\left\langle f^h_k,\pi^h\right\rangle}=\left\langle f^h,\pi^h\right\rangle-\lim_{k\to+\infty}{\sum_{|i|>k}{f_i\pi_i}}=\left\langle f^h,\pi^h\right\rangle.
\end{equation*}
Therefore, we can establish the Poincar\'e inequality \eqref{GDPI} for all $f^h\in\mathcal{H}^1(\pi^h)$ with constant $\kappa=\widetilde{\lambda}$.

\subsection{Curvature condition}\label{S2.2}
In this subsection we aim to obtain the curvature condition \eqref{CC}. By definition we can get the following result, and we leave the proof to subsection \ref{S2.3}.
\begin{lemma}\label{G2}
    For $f^h,g^h\in\ell^2(\pi^h)$, it holds that
    \begin{align*}
        \left(\Gamma^h_2(f^h,g^h)\right)_i
        &=\frac{1}{4}\alpha_i\left(3\Delta^+\beta_i-\Delta^+\alpha_i\right)\Delta^+f_i\Delta^+g_i+\frac{1}{4}\beta_i(\Delta^-\beta_i-3\Delta^-\alpha_i)\Delta^-f_i\Delta^-g_i\\
        &\quad+\frac{1}{4}\alpha_i\alpha_{i+1}\Delta^+\Delta^+f_i\Delta^+\Delta^+g_i+\frac{1}{4}\beta_i\beta_{i-1}\Delta^-\Delta^-f_i\Delta^-\Delta^-g_i\\
        &\quad+\frac{1}{2}\alpha_i\beta_i\Delta^+\Delta^-f_i\Delta^+\Delta^-g_i,
    \end{align*}
\end{lemma}
In particular, for all $f^h\in\ell^2(\pi^h)$ it holds that
\begin{equation*}
    \left(\Gamma^h(f^h,f^h)\right)_i=\frac{h}{2}\alpha_i(\Delta^+f_i)^2+\frac{h}{2}\beta_i(\Delta^-f_i)^2,
\end{equation*}
and
\begin{equation}
    \left(\Gamma^h_2(f^h,f^h)\right)_i\geq \frac{1}{4}\alpha_i(3\Delta^+\beta_i-\Delta^+\alpha_i)(\Delta^+f_i)^2+\frac{1}{4}\beta_i(\Delta^-\beta_i-3\Delta^-\alpha_i)(\Delta^-f_i)^2.
\end{equation}
If we can obtain a uniform positive lower bound for both $3\Delta^+\beta_i-\Delta^+\alpha_i$ and $\Delta^-\beta_i-3\Delta^-\alpha_i$, then the curvature condition is satisfied. We define the coefficient function as
\begin{equation*}
    \alpha(s)=\frac{1}{h^2}B\left(u(s+h)-u(s)\right).
\end{equation*}
Then $\alpha(ih)=\alpha_i$ for all $i\in\mathbb{Z}$, and this definition is consistent with \eqref{C}. It follows that
\begin{align*}
    \alpha'(s)
    &=\frac{1}{h^2}B'\left(u(s+h)-u(s)\right)\left(u'(s+h)-u'(s)\right)\\
    &=\frac{1}{h}B'\left(u(s+h)-u(s)\right)\int^1_0{u''(s+wh)}dw,
\end{align*}
and
\begin{align*}
    \alpha(ih+h)-\alpha(ih)
    &=h\int^1_0{\alpha'(ih+vh)}dv\\
    &=\int^1_0{\int^1_0{B'\left(u(ih+vh+h)-u(ih+vh)\right)u''(ih+vh+wh)}dv}dw.
\end{align*}
In fact by simple calculation one can obtain the following assertions
\begin{equation*}
    \left\{\begin{aligned}
        &\Delta^+\alpha_i=\int^1_0{\int^1_0{B'\left(u(ih+vh+h)-u(ih+vh)\right)u''(ih+vh+wh)}dv}dw,\\
        &\Delta^+\beta_i=-\int^1_0{\int^1_0{B'(u(ih+vh-h)-u(ih+vh))u''(ih+vh-wh)}dv}dw,\\
        &\Delta^-\alpha_i=\int^1_0{\int^1_0{B'(u(ih-vh+h)-u(ih-vh))u''(ih-vh+wh)}dv}dw,\\
        &\Delta^-\beta_i=-\int^1_0{\int^1_0{B'(u(ih-vh-h)-u(ih-vh))u''(ih-vh-wh)}dv}dw.
    \end{aligned}
    \right.
\end{equation*}
We emphasize that $i$ is always an integer in this work. Since by assumption $u''(s)\geq\lambda$ and $B'(s)<0$ for all $s\in\mathbb{R}$, it holds that
\begin{equation*}
    \left\{\begin{aligned}
        &3\Delta^+\beta_i-\Delta^+\alpha_i\geq\\
        &\quad-\lambda\int^1_0{\left[3B'(u(ih+vh-h)-u(ih+vh))+B'(u(ih+vh+h)-u(ih+vh))\right]}dv,\\
        &\Delta^-\beta_i-3\Delta^-\alpha_i\geq\\
        &\quad-\lambda\int^1_0{\left[B'(u(ih-vh-h)-u(ih-vh))+3B'(u(ih-vh+h)-u(ih-vh))\right]}dv.
    \end{aligned}
    \right.
\end{equation*}

When $u$ is convex, there exists a ball $B(0,R)$ such that $u$ is monotone outside $B(0,R)$. Here we may abuse the
symbols, and whether $B$ denotes a function or a ball should be clear.  For sufficiently large $|i|$, the following inequalities hold for aribitrary $v\in (0,1)$
\begin{equation*}
    \left\{\begin{aligned}
        &\left[u(ih+vh-h)-u(ih+vh)\right]\left[u(ih+vh+h)-u(ih+vh)\right]\leq0,\\
        &\left[u(ih-vh+h)-u(ih-vh)\right]\left[u(ih-vh-h)-u(ih-vh)\right]\leq0.
    \end{aligned}
    \right.
\end{equation*}
Since by assumption the function $B'(s)$ is non-positive and monotonically increasing, then for arbitrary $s<0$ it holds that $B'(s)\leq B'(0)=-1/2$, that is,
\begin{equation*}
    \left\{\begin{aligned}
        &-3B'(u(ih+vh-h)-u(ih+vh))-B'(u(ih+vh+h)-u(ih+vh))\geq\frac{1}{2},\\
        &-B'(u(ih-vh-h)-u(ih-vh))-3B'(u(ih-vh+h)-u(ih-vh))\geq\frac{1}{2}.
    \end{aligned}
    \right.
\end{equation*}
Then for points outside $B(0,R)$ it holds that
\begin{equation*}
    \left(\Gamma^h_2(f^h,f^h)\right)_i\geq\frac{\lambda}{8}\left(\Gamma^h(f^h,f^h)\right)_i.
\end{equation*}
And for the rest points within $B(0,R)$, both $u(s)$ and $B'(s)$ are uniformly bounded by continuity. Hence, there exists a constant $\widetilde{\lambda}>0$ such that
\begin{equation*}
    \Gamma^h_2(f^h,f^h)\geq\widetilde{\lambda}\Gamma^h(f^h,f^h).
\end{equation*}

\begin{remark}
    The constant $\widetilde{\lambda}$ is independent of $h$. This uniformity is essential for establishing convergence results that are valid across different discretization levels.
\end{remark}

\begin{remark}
    Recall that $\lim_{s\to0}B'(s)=-1/2$ by definition. Then
    \begin{equation*}
    \left\{\begin{aligned}
        &\lim_{h\to0}{-\frac{1}{4}\left[3B'(u(ih+vh-h)-u(ih+vh))+B'(u(ih+vh+h)-u(ih+vh))\right]}=\frac{1}{2},\\
        &\lim_{h\to0}{-\frac{1}{4}\left[B'(u(ih-vh-h)-u(ih-vh))+3B'(u(ih-vh+h)-u(ih-vh))\right]}=\frac{1}{2}.
    \end{aligned}
    \right.
    \end{equation*}
    that is,
    \begin{equation*}
        \lim_{h\to0}{\left(\Gamma^h_2(f^h,f^h)\right)_i}\geq\lambda\lim_{h\to0}{\left(\Gamma^h(f^h,f^h)\right)_i},
    \end{equation*}
    which is consistent with the continuous case.
\end{remark}

\subsection{Calculation of $\Gamma^h_2$}\label{S2.3}
In this subsection we derive the expression for $\Gamma^h_2(f^h,g^h)$ and give a proof to Lemma \ref{G2}. Recall that
\begin{equation*}
    \Gamma^h(f^h,g^h)=\frac{1}{2}\alpha_i\Delta^+f_i\Delta^+g_i+\frac{1}{2}\Delta^-f_i\Delta^-g_i,
\end{equation*}
and
\begin{equation*}
    \Gamma^h_2(f_i,g_i)=\frac{1}{2}\mathcal{L}_h\Gamma^h(f_i,g_i)-\frac{1}{2}\Gamma^h(f_i,\mathcal{L}_hg_i)-\frac{1}{2}\Gamma^h(\mathcal{L}_hf_i,g_i).
\end{equation*}
By definition one has
\begin{equation*}
    \left\{\begin{aligned}
        &\begin{aligned}
            \mathcal{L}_h\Gamma^h(f_i,g_i)
            &=\frac{1}{2}\alpha_i\Delta^+\left(\alpha_i\Delta^+f_i\Delta^+g_i\right)-\frac{1}{2}\beta_i\Delta^-\left(\beta_i\Delta^-f_i\Delta^-g_i\right)\\
            &\quad+\frac{1}{2}\alpha_i\Delta^+\left(\beta_i\Delta^-f_i\Delta^-g_i\right)-\frac{1}{2}\beta_i\Delta^-\left(\alpha_i\Delta^+f_i\Delta^+g_i\right),
        \end{aligned}\\
        &\begin{aligned}
            \Gamma^h(f_i,\mathcal{L}_hg_i)
            &=\frac{1}{2}\alpha_i\Delta^+f_i\Delta^+\left(\alpha_i\Delta^+g_i\right)-\frac{1}{2}\beta_i\Delta^-f_i\Delta^-\left(\beta_i\Delta^-g_i\right)\\
            &\quad-\frac{1}{2}\alpha_i\Delta^+f_i\Delta^+\left(\beta_i\Delta^-g_i\right)+\frac{1}{2}\beta_i\Delta^-f_i\Delta^-\left(\alpha_i\Delta^+g_i\right),
        \end{aligned}\\
        &\begin{aligned}
            \Gamma^h(\mathcal{L}_hf_i,g_i)
            &=\frac{1}{2}\alpha_i\Delta^+g_i\Delta^+\left(\alpha_i\Delta^+f_i\right)-\frac{1}{2}\beta_i\Delta^-g_i\Delta^-\left(\beta_i\Delta^-f_i\right)\\
            &\quad-\frac{1}{2}\alpha_i\Delta^+g_i\Delta^+\left(\beta_i\Delta^-f_i\right)+\frac{1}{2}\beta_i\Delta^-g_i\Delta^-\left(\alpha_i\Delta^+f_i\right).
        \end{aligned}
    \end{aligned}
    \right.
\end{equation*}
Then there are four distinct terms in each of the expressions $\mathcal{L}_h\Gamma^h(f_i,g_i)$, $\Gamma^h(f_i,\mathcal{L}_hg_i)$, and $\Gamma^h(\mathcal{L}_hf_i,g_i)$. Focusing on their first terms, we find
\begin{align*}
    &\alpha_i\Delta^+\left(\alpha_i\Delta^+f_i\Delta^+g_i\right)-\alpha_i\Delta^+f_i\Delta^+\left(\alpha_i\Delta^+g_i\right)-\alpha_i\Delta^+g_i\Delta^+\left(\alpha_i\Delta^+f_i\right)\\
    &\qquad=\alpha_i\alpha_{i+1}\Delta^+\left(\Delta^+f_i\Delta^+g_i\right)-\alpha_i\alpha_{i+1}\Delta^+f_i\Delta^+\Delta^+g_i-\alpha_i\alpha_{i+1}\Delta^+g_i\Delta^+\Delta^+f_i\\
    &\qquad\quad+\alpha_i\Delta^+\alpha_i\Delta^+f_i\Delta^+g_i-\alpha_i\Delta^+\alpha_i\Delta^+f_i\Delta^+g_i-\alpha_i\Delta^+g_i\Delta^+f_i.
\end{align*}
Note that $\Delta^+(f_ig_i)-f_i\Delta^+g_i-g_i\Delta^+f_i=\Delta^+f_i\Delta^+g_i$. If we replace $f_i$ by $\Delta^+f_i$ and $g_i$ by $\Delta^+g_i$, it holds that
\begin{equation*}
    \Delta^+\left(\Delta^+f_i\Delta^+g_i\right)-\Delta^+f_i\Delta^+\Delta^+g_i-\Delta^+g_i\Delta^+\Delta^+f_i=\Delta^+\Delta^+f_i\Delta^+\Delta^+g_i,
\end{equation*}
that is,
\begin{align*}
    &\alpha_i\Delta^+\left(\alpha_i\Delta^+f_i\Delta^+g_i\right)-\alpha_i\Delta^+f_i\Delta^+\left(\alpha_i\Delta^+g_i\right)-\alpha_i\Delta^+g_i\Delta^+\left(\alpha_i\Delta^+f_i\right)\\
    &\qquad=\alpha_i\alpha_{i+1}\Delta^+\Delta^+f_i\Delta^+\Delta^+g_i-\alpha_i\Delta^+\alpha_i\Delta^+f_i\Delta^+g_i.
\end{align*}
Here we can get parallel results for the other terms in the same way. Indeed one can have
\begin{equation*}
    \left\{\begin{aligned}
        &\begin{aligned}
            &\alpha_i\Delta^+\left(\alpha_i\Delta^+f_i\Delta^+g_i\right)-\alpha_i\Delta^+f_i\Delta^+\left(\alpha_i\Delta^+g_i\right)-\alpha_i\Delta^+g_i\Delta^+\left(\alpha_i\Delta^+f_i\right)\\
            &\qquad=\alpha_i\alpha_{i+1}\Delta^+\Delta^+f_i\Delta^+\Delta^+g_i-\alpha_i\Delta^+\alpha_i\Delta^+f_i\Delta^+g_i,
        \end{aligned}\\
        &\begin{aligned}
            &-\beta_i\Delta^-\left(\beta_i\Delta^-f_i\Delta^-g_i\right)+\beta_i\Delta^-f_i\Delta^-\left(\alpha_i\Delta^-g_i\right)+\beta_i\Delta^-g_i\Delta^-\left(\alpha_i\Delta^-f_i\right)\\
            &\qquad=\beta_i\beta_{i-1}\Delta^-\Delta^-f_i\Delta^-\Delta^-g_i+\beta_i\Delta^-\beta_i\Delta^-f_i\Delta^-g_i,
        \end{aligned}\\
        &\begin{aligned}
            &\alpha_i\Delta^+\left(\beta_i\Delta^-f_i\Delta^-g_i\right)+\alpha_i\Delta^+f_i\Delta^+\left(\beta_i\Delta^-g_i\right)+\alpha_i\Delta^+g_i\Delta^+\left(\beta_i\Delta^-f_i\right)\\
            &\qquad=\alpha_i\beta_i\Delta^+\Delta^-f_i\Delta^+\Delta^-g_i+3\alpha_i\Delta^+\beta_i\Delta^+f_i\Delta^+g_i,
        \end{aligned}\\
        &\begin{aligned}
            &-\beta_i\Delta^-\left(\alpha_i\Delta^+f_i\Delta^+g_i\right)-\beta_i\Delta^-f_i\Delta^-\left(\alpha_i\Delta^+g_i\right)-\beta_i\Delta^-g_i\Delta^-\left(\alpha_i\Delta^+f_i\right)\\
            &\qquad=\alpha_i\beta_i\Delta^+\Delta^-f_i\Delta^+\Delta^-g_i-3\beta_i\Delta^-\alpha_i\Delta^-f_i\Delta^-g_i.
        \end{aligned}
    \end{aligned}
    \right.
\end{equation*}
And we can obtain the expression of $\Gamma^h_2(f^h,g^h)$,
\begin{align*}
    \left(\Gamma^h_2(f^h,g^h)\right)_i
    &=\frac{1}{4}\alpha_i\left(3\Delta^+\beta_i-\Delta^+\alpha_i\right)\Delta^+f_i\Delta^+g_i+\frac{1}{4}\beta_i(\Delta^-\beta_i-3\Delta^-\alpha_i)\Delta^-f_i\Delta^-g_i\\
    &\quad+\frac{1}{4}\alpha_i\alpha_{i+1}\Delta^+\Delta^+f_i\Delta^+\Delta^+g_i+\frac{1}{4}\beta_i\beta_{i-1}\Delta^-\Delta^-f_i\Delta^-\Delta^-g_i\\
    &\quad+\frac{1}{2}\alpha_i\beta_i\Delta^+\Delta^-f_i\Delta^+\Delta^-g_i.
\end{align*}

\subsection{Corollaries of Theorem \ref{T1}}
In this subsection, we derive two corollaries of Theorem \ref{T1}. At first we establish Corollary \ref{C1}, which claims that $P_tf^h$ converges to zero in $\mathcal{H}(\pi^h)$ exponentially fast when $u$ is strongly convex,
\begin{equation*}
    \begin{aligned}
        &\left\langle\left(P_tf^h-\langle P_tf^h,\pi^h\rangle\right)^2+\Gamma^h(P_tf^h,P_tf^h),\pi^h\right\rangle\\
        &\qquad\leq e^{-\theta t}\left\langle\left(f^h-\langle f^h,\pi^h\rangle\right)^2+\Gamma^h(f^h,f^h),\pi^h\right\rangle.
    \end{aligned}
\end{equation*}
where $\theta$ is a positive constant and $f^h\in\ell^\infty$.

Having established the exponential decay of $\Gamma^h(P_tf^h,P_tf^h)$ in the proof of Theorem \ref{T1}, the result follows directly from Proposition \ref{T} combined with Gr\"onwall's inequality. Following the idea in \cite{baudoin2017bakry}, we present an alternative approach to proving the exponential convergence, which might be potentially extended to the semi-convex potential $u$.

\begin{proof}[Proof of Corollary \ref{C1}]
    By \eqref{CoG} one can obtain
    \begin{equation}\label{CoG2}
        \left\langle\Gamma^h(P_tf^h,P_tf^h),\pi^h\right\rangle\leq e^{-2\widetilde{\lambda}t}\left\langle\Gamma^h(f^h,f^h),\pi^h\right\rangle.
    \end{equation}
    And by Proposition \ref{P1} it holds that
    \begin{equation}\label{II}
        P_t\left[(P_{T-t}f^h)^2\right]-(P_Tf^h)^2=2\int^t_0{P_s\left[\Gamma^h(P_{T-s}f^h,P_{T-s}f^h)\right]}ds.
    \end{equation}
    Using Fubini's theorem and \eqref{PI1}, one can get
    \begin{align*}
        \left\langle\int^t_0{P_s\left[\Gamma^h(P_{T-s}f^h,P_{T-s}f^h)\right]}ds,\pi^h\right\rangle
        &=\int^t_0{\left\langle P_s\left[\Gamma^h(P_{T-s}f^h,P_{T-s}f^h)\right],\pi^h\right\rangle}ds\\
        &=\int^t_0{\left\langle\Gamma^h(P_{T-s}f^h,P_{T-s}f^h),\pi^h\right\rangle}ds\\
        &\geq\int^t_0{\widetilde{\lambda}\left\langle\left(P_{T-s}f^h-\left\langle P_{T-s}f^h,\pi^h\right\rangle\right)^2,\pi^h\right\rangle}ds\\
        &=\widetilde{\lambda}\int^t_0{\left\langle P_s\left[\left(P_{T-s}f^h-\left\langle P_{T-s}f^h,\pi^h\right\rangle\right)^2\right],\pi^h\right\rangle}ds.
    \end{align*}
    Since $\langle P_tf^h,\pi^h\rangle=\langle f^h,\pi^h\rangle$ is a constant, one can have
    \begin{align*}
        &\left\langle P_t\left[(P_{T-t}f^h)^2\right]-(P_Tf^h)^2,\pi^h\right\rangle\\
        &\qquad=\left\langle(P_{T-t}f^h)^2,\pi^h\right\rangle-\left\langle P_{T-t}f^h,\pi^h\right\rangle^2-\left\langle(P_Tf^h)^2,\pi^h\right\rangle+\left\langle P_Tf^h,\pi^h\right\rangle^2\\
        &\qquad=\left\langle\left(P_{T-t}f^h-\left\langle P_{T-t}f^h,\pi^h\right\rangle\right)^2,\pi^h\right\rangle-\left\langle\left(P_Tf^h-\left\langle P_Tf^h,\pi^h\right\rangle\right)^2,\pi^h\right\rangle\\
        &\qquad=\left\langle P_t\left[\left(P_{T-t}f^h-\left\langle P_{T-t}f^h,\pi^h\right\rangle\right)^2\right],\pi^h\right\rangle-\left\langle\left(P_Tf^h-\left\langle P_Tf^h,\pi^h\right\rangle\right)^2,\pi^h\right\rangle.
    \end{align*}
    Then by integrating both sides of \eqref{II} with respect to $\pi^h$,
    \begin{align*}
        &\left\langle P_t\left[\left(P_{T-t}f^h-\left\langle P_{T-t}f^h,\pi^h\right\rangle\right)^2\right],\pi^h\right\rangle-\left\langle\left(P_Tf^h-\left\langle P_Tf^h,\pi^h\right\rangle\right)^2,\pi^h\right\rangle\\
        &\qquad\geq2\widetilde{\lambda}\int^t_0{\left\langle P_s\left[\left(P_{T-s}f^h-\left\langle P_{T-s}f^h,\pi^h\right\rangle\right)^2\right],\pi^h\right\rangle}ds,
    \end{align*}
    and it follows by Gr\"onwall's inequality that
    \begin{equation*}
        \left\langle P_T\left[\left(f^h-\left\langle f^h,\pi^h\right\rangle\right)^2\right],\pi^h\right\rangle\geq e^{2\widetilde{\lambda}T}\left\langle\left(P_Tf^h-\left\langle P_Tf^h,\pi^h\right\rangle\right)^2,\pi^h\right\rangle,
    \end{equation*}
    that is,
    \begin{equation}
        \left\langle(P_tf^h)^2,\pi^h\right\rangle-\left\langle P_tf^h,\pi^h\right\rangle^2\leq e^{-2\widetilde{\lambda}t}\left(\left\langle(f^h)^2,\pi^h\right\rangle-\left\langle f^h,\pi^h\right\rangle^2\right)
    \end{equation}
    Together with \eqref{CoG2} we can conclude that
    \begin{equation*}
        \begin{aligned}
            &\left\langle\left(P_tf^h-\langle P_tf^h,\pi^h\rangle\right)^2+\Gamma^h(P_tf^h,P_tf^h),\pi^h\right\rangle\\
            &\qquad\leq e^{-2\widetilde{\lambda} t}\left\langle\left(f^h-\langle f^h,\pi^h\rangle\right)^2+\Gamma^h(f^h,f^h),\pi^h\right\rangle.
        \end{aligned}
    \end{equation*}
\end{proof}

Next we make use of a perturbation argument to obtain Corollary \ref{C2}. Notice that the operator $\Gamma^h$ and the coefficients $\{\alpha_i,\beta_i\}_{i\in\mathbb{Z}}$ exist in exact bijective correspondence. Since the detailed balance condition holds, then the operator $\Gamma^h$ is uniquely determined by the stationary solution $\pi$ of the Fokker-Planck equation \eqref{FPE}. This relationship becomes crucial when analyzing perturbations, as modifying the potential $u$ simultaneously affects both the stationary measure $\pi^h$ and the associated operator $\Gamma^h$.

\begin{proof}[Proof of Corollary \ref{C2}]
    Let $\widetilde{\pi}$ be the stationary solution with respect to $\widetilde{u}$, the corresponding coefficients $\widetilde{\alpha}_i,\widetilde{\beta}_i$ and the related \textit{carr\'e du champ} operator $\widetilde{\Gamma}^h$. Since by assumption $\widetilde{\pi}^h$ differs $\pi^h$ on a compact set, then it holds that
    \begin{equation*}
        0<\left\|\frac{\pi^h}{\widetilde{\pi}^h}\right\|^{-1}_{\ell^\infty}\leq\frac{\widetilde{\pi}_i}{\pi_i}\leq\left\|\frac{\widetilde{\pi}^h}{\pi^h}\right\|_{\ell^\infty}<+\infty,\quad\forall\ i\in\mathbb{Z},
    \end{equation*}
    and
        \begin{equation*}
        \sup_{i\in\mathbb{Z}}{\frac{\widetilde{\alpha}_i}{\alpha_i}}<+\infty,\qquad \sup_{i\in\mathbb{Z}}{\frac{\widetilde{\beta}_i}{\beta_i}}<+\infty.
    \end{equation*}
    Then we can claim that $\mathcal{H}^1(\pi^h)=\mathcal{H}^1(\widetilde{\pi}^h)$.

    Although $\widetilde{\pi}^h$ may not be a probability, we can still make use of Theorem \ref{T1} to obtain the discrete Poincar\'e inequality for $\widetilde{\pi}^h$ with a positive constant $\kappa$, since the inequality \eqref{GDPI} does not change under scaling of the measure. Consequently, for any constant $m$ and $f^h\in\mathcal{H}^1(\pi^h)$ it holds that
    \begin{align*}
        \left\langle(f^h-m)^2,\pi^h\right\rangle
        &\leq\left\|\frac{\pi^h}{\widetilde{\pi}^h}\right\|_{\ell^\infty}\left\langle(f^h-m)^2,\widetilde{\pi}^h\right\rangle\\
        &\leq\frac{1}{\widetilde{\kappa}}\left\|\frac{\pi^h}{\widetilde{\pi}^h}\right\|_{\ell^\infty}\left\langle\widetilde{\Gamma}^h(f^h,f^h),\widetilde{\pi}^h\right\rangle\\
        &\leq\frac{1}{\widetilde{\kappa}}\left\|\frac{\pi^h}{\widetilde{\pi}^h}\right\|_{\ell^\infty}\left\|\frac{\widetilde{\pi}^h}{\pi^h}\right\|_{\ell^\infty}\left\langle\widetilde{\Gamma}^h(f^h,f^h),\pi^h\right\rangle\\
        &\leq\frac{1}{\widetilde{\kappa}}\left\|\frac{\pi^h}{\widetilde{\pi}^h}\right\|_{\ell^\infty}\left\|\frac{\widetilde{\pi}^h}{\pi^h}\right\|_{\ell^\infty}\sup_{i\in\mathbb{Z}}{\left(\frac{\widetilde{\alpha}_i}{\alpha_i}+\frac{\widetilde{\beta}_i}{\beta_i}\right)}\left\langle\Gamma^h(f^h,f^h),\pi^h\right\rangle.
    \end{align*}
    Set $m=\langle f^h,\pi^h\rangle$, then we can obtain the Poincar\'e inequality for $\pi^h$ and \textit{Corollary \ref{C2}}.
\end{proof}

\section{Discrete Poincar\'e inequality by Lyapunov function}\label{S3}
In this section, we introduce another approach to establish the discrete Poincar\'e inequality. We make use of the Lyapunov function to extend a local discrete Poincar\'e inequality to the full space. The function $W$ is called a Lyapunov function if $W\geq1$ and if there exists $\theta>0$, $b\geq0$ and some $R>0$ such that the following inequality holds for all $i\in\mathbb{Z}$
\begin{equation}\label{LF}
    \mathcal{L}_hW_i\leq-\theta W_i+bI_{B(0,R)}(x_i),
\end{equation}
where $W_i:=W(x_i)=W(ih)$ and $B(0,R)$ is a ball centered at $0$ with radius $R$. And we define $\pi^h|_{B(0,R)}$ as a conditional distribution,
\begin{equation}
    \pi^h|_{B(0,R)}(A):=\frac{\pi^h(A\bigcap B(0,R))}{\pi^h(B(0,R))},
\end{equation}
where $A$ is a Borel measurable set in $\mathbb{R}$. It can be verified that $\mathcal{H}^1(\pi^h)\subset\mathcal{H}^1(\pi^h|_{B(0,R)})$, then we may get the discrete Poincar\'e inequality for $\pi^h$ through that of $\pi^h|_{B(0,R)}$.

Then we can introduce the key idea in this section, which is an analogous result of \cite[Theorem 1.4]{bakry2008rate}. The proof is left in subsection \ref{S3.1}.

\begin{lemma}\label{L0}
    Suppose there exists a Lyapunov function $W\geq1$ satisfying \eqref{LF}. If the conditional measure $\pi^h|_{B(0,R)}$ satisfies a local discrete discrete Poincar\'e inequality on the ball $B(0,R)$, then $\pi^h$ will satisfy a global discrete Pioncar\'e inequality in the full space.
\end{lemma}

To establish Theorem \ref{T2}, we need to construct a Lyapunov function $W\geq$ satisfying \eqref{LF}, and verify a local discrete Poincar\'e inequality for $\pi^h$. In subsection \ref{S3.2}, we will demonstrate the existence of a suitable Lyapunov function $W$ satisfying \eqref{LF} for sufficiently small $h$.

\begin{lemma}\label{L1}
    Under Assumption \ref{A0} and \ref{ass:B}, there exists $h_0>0$ such that $W(x):=e^{|x|}$ is a Lyapunov function satisfying \eqref{LF} for all $0<h<h_0$.
\end{lemma}

In the continuous setting, when the potential $u$ is locally bounded, the conditional measure $\pi|_{B(0,R)}$ is known to satisfy a local Poincar\'e inequality. We will establish a parallel result for the discrete landscape in subsection \ref{S3.3}.

\begin{lemma}\label{L2}
    For any radius $R>0$ the conditional measure $\pi^h|_{B(0,R)}$ satisfies a local discrete Poincar\'e inequality with constant $\kappa_R$ in the sense that the following inequality holds for all $f^h\in\mathcal{H}^1(\pi^h|_{B(0,R)})$,
    \begin{equation}\label{LPI}
        \left\langle(f^h)^2,\pi^h|_{B(0,R)}\right\rangle-\left\langle f^h,\pi^h|_{B(0,R)}\right\rangle^2\leq\frac{1}{\kappa_R}\left\langle\Gamma^h(f^h,f^h),\pi^h|_{B(0,R)}\right\rangle.
    \end{equation}
\end{lemma}

With Lemma \ref{L0}, Lemma \ref{L1} and Lemma \ref{L2}, one can have
\begin{equation}\label{LDPI}
    \left\langle(f^h)^2,\pi^h\right\rangle-\left\langle f^h,\pi^h\right\rangle^2\leq\frac{1}{\kappa}\left\langle\Gamma^h(f^h,f^h),\pi^h\right\rangle,
\end{equation}
where $f^h\in\mathcal{H}^1(\pi^h)$. And we have established Theorem \ref{T2}.

\begin{remark}
    Under Assumption \ref{A0} and \ref{ass:B}, the local discrete Poincar\'e inequality on $B(0,R)$ with constant $\kappa_R$ can be extended to a global inequality with constant $\kappa=\theta\kappa_R/(\kappa_R+b)$.
\end{remark}

\subsection{Proof of Lemma \ref{L0}}\label{S3.1}
In this subsection, we establish Theorem \ref{T2} using the Lyapunov function approach. Following the idea in \cite{bakry2008rate}, we begin with a key observation that applies to arbitrary sequences $W^h$, and which remains valid when $W$ is specifically chosen as a Lyapunov function. We derive the following properties:

\begin{proposition}\label{P2}
    For any $f^h\in\ell^2(\pi^h)$ and arbitrary sequence $W^h$, it holds that
    \begin{equation}\label{fw}
        \begin{aligned}
            &\Gamma^h\left(\frac{(f_i)^2}{W_i},W_i\right)-\Gamma^h(f_i,f_i)\\
            &\quad=-\frac{1}{2}\int^1_0{\left[\alpha_i\left(\frac{f_{i+1}W_i-f_iW_{i+1}}{W_i+s\Delta^+W_i}\right)^2+\beta_i\left(\frac{f_{i-1}W_i-f_iW_{i-1}}{W_i-s\Delta^-W_i}\right)^2\right]}ds,
        \end{aligned}
    \end{equation}
    then
    \begin{equation}
        \Gamma^h\left(\frac{(f^h)^2}{W^h},W^h\right)-\Gamma^h(f^h,f^h)\leq0.
    \end{equation}
\end{proposition}

\begin{proof}
    Define $\widetilde{f}_i(s):=f_i+s\Delta^+f_i$ and $\widetilde{W}_i(s):=W_i+s\Delta^+W_i$. Then both $\widetilde{f}'(s)=\Delta^+f_i$ and $\widetilde{W}'(s)=\Delta^+W_i$ are constants, and it holds that
    \begin{align*}
        \Delta^+\left(\frac{f^2_i}{W_i}\right)\Delta^+W_i
        &=\int^1_0{\left(\frac{\widetilde{f}^2_i}{\widetilde{W}_i}\right)'(s)\widetilde{W}'(s)}ds\\
        &=\int^1_0{\left(\frac{2\widetilde{f}_i\widetilde{f}'_i}{\widetilde{W}_i}\widetilde{W}'_i(s)-\frac{\widetilde{f}^2_i}{\widetilde{W}^2_i}\left(\widetilde{W}'_i\right)^2(s)\right)}ds\\
        &=\int^1_0{\left(\widetilde{f}'_i\right)^2(s)}ds-\int^1_0{\left(\widetilde{f}'_i-\frac{\widetilde{f}_i}{\widetilde{W}_i}\widetilde{W}'_i\right)^2(s)}ds\\
        &=\left(\Delta^+f_i\right)^2-\int^1_0{\left(\Delta^+f_i-\frac{f_i+s\Delta^+f_i}{W_i+s\Delta^+W_i}\Delta^+W_i\right)^2}ds.
    \end{align*}
    By a parallel argument one can have
    \begin{equation*}
        \Delta^-\left(\frac{f^2_i}{W_i}\right)\Delta^-W_i=(\Delta^-f_i)^2-\int^1_0{\left(\Delta^-f_i-\frac{f_i-s\Delta^-f_i}{W_i-s\Delta^-W_i}\Delta^-W_i\right)^2}ds.
    \end{equation*}
    And together with \eqref{G} one can obtain the identity \eqref{fw}.
\end{proof}

Now we return to Lemma \ref{L0}. By \eqref{LF} it holds that
\begin{equation*}
    1\leq-\frac{\mathcal{L}_hW_i}{\theta W_i}+\frac{b}{\theta}\frac{I_{B(0,R)}(x_i)}{W_i}\leq-\frac{\mathcal{L}_hW_i}{\theta W_i}+\frac{b}{\theta}I_{B(0,R)}(x_i),
\end{equation*}
and for arbitrary $f^h\in\mathcal{H}^1(\pi^h)$,
\begin{align*}
    \left\langle(f^h)^2,\pi^h\right\rangle
    &\leq\left\langle-\frac{\mathcal{L}_hW^h}{\theta W^h}(f^h)^2+\frac{b}{\theta}I^h_{B(0,R)}(f^h)^2,\pi^h\right\rangle\\
    &=\frac{1}{\theta}\left\langle-\frac{(f^h)^2}{W^h}\mathcal{L}_hW^h,\pi^h\right\rangle+\frac{b}{\theta}\pi^h(B(0,R))\left\langle(f^h)^2,\pi^h|_{B(0,R)}\right\rangle.
\end{align*}
Since $\mathcal{L}_h$ is symmetric with respect to $\pi^h$ and $\mathcal{L}^*_h\pi^h=0$, one gets
\begin{align*}
    \left\langle-\frac{(f^h)^2}{W^h}\mathcal{L}_hW^h,\pi^h\right\rangle
    &=\frac{1}{2}\left\langle(f^h)^2,\mathcal{L}^*_h\pi^h\right\rangle-\frac{1}{2}\left\langle\frac{(f^h)^2}{W^h}\mathcal{L}_hW^h,\pi^h\right\rangle-\frac{1}{2}\left\langle W^h\mathcal{L}_h\left(\frac{(f^h)^2}{W^h}\right),\pi^h\right\rangle\\
    &=\frac{1}{2}\left\langle\mathcal{L}_h\left(\frac{(f^h)^2}{W^h}W^h\right)-\frac{(f^h)^2}{W^h}\mathcal{L}_hW^h-W^h\mathcal{L}_h\left(\frac{(f^h)^2}{W^h}\right),\pi^h\right\rangle\\
    &=\left\langle\Gamma^h\left(\frac{(f^h)^2}{W^h},W^h\right),\pi^h\right\rangle\\
    &\leq\left\langle\Gamma^h(f^h,f^h),\pi^h\right\rangle.
\end{align*}
The last inequality follows directly from Proposition \ref{P2}. Note that $\Gamma^h(f^h-m,f^h-m)=\Gamma^h(f^h,f^h)$ for arbitrary constant $m\in\mathbb{R}$. Then one can have
\begin{equation*}
    \left\langle-\frac{(f^h-m)^2}{W^h}\mathcal{L}_hW^h,\pi^h\right\rangle\leq\left\langle\Gamma^h(f^h,f^h),\pi^h\right\rangle.
\end{equation*}
It follows that
\begin{equation}
    \left\langle(f^h-m)^2,\pi^h\right\rangle\leq\frac{1}{\theta}\left\langle\Gamma^h(f^h,f^h),\pi^h\right\rangle+\frac{b}{\theta}\pi^h(B(0,R))\left\langle(f^h-m)^2,\pi^h|_{B(0,R)}\right\rangle.
\end{equation}
If we set $m=\langle f^h,\pi^h|_{B(0,R)}\rangle$, by the local Poincar\'e inequality \eqref{LPI},
\begin{align*}
    \pi^h(B(0,R))\left\langle(f^h-m)^2,\pi^h|_{B(0,R)}\right\rangle
    &\leq\frac{1}{\kappa_R}\pi^h(B(0,R))\left\langle\Gamma^h(f^h,f^h),\pi^h|_{B(0,R)}\right\rangle\\
    &\leq\frac{1}{\kappa_R}\left\langle\Gamma^h(f^h,f^h),\pi^h\right\rangle.
    \end{align*}
Hence, we can conclude that
\begin{align*}
    \left\langle(f^h)^2,\pi^h\right\rangle-\left\langle f^h,\pi^h\right\rangle^2
    &\leq\left\langle(f^h-m)^2,\pi^h\right\rangle\\
    &\leq\frac{\kappa_R+b}{\theta\kappa_R}\left\langle\Gamma^h(f^h,f^h),\pi^h\right\rangle,
\end{align*}
that is, by letting $\kappa=\theta\kappa_R/(\kappa_R+b)$, one can get the inequality \eqref{LDPI}.

\subsection{Construction of Lyapunov function}\label{S3.2}
In this subsection we give a proof to Lemma \ref{L1}. We firstly discuss $\mathcal{L}_hW_i/W_i$ for sufficiently large $|i|$, and we seek for the constants $\theta$ and $R$ such that $\mathcal{L}_hW_i\leq-\theta W_i$ for all $|x_i|>R$. And for the remaining points within the ball $B(0,R)$, the continuity of $B(s)$ and $u(x)$ guarantees the existence of a constant $b$ satisfying \eqref{LF}. Let
\begin{equation}
    W(x):=e^{|x|},
\end{equation}
and it is clear that $W\geq1$.

Let $i$ be a positive integer. Recall that $x_i=ih$, $W_i:=W(x_i)$ and $u_i:=u(x_i)$. Then
\begin{align*}
    \frac{\mathcal{L}_hW_i}{W_i}
    &=\alpha_i\left(e^h-1\right)+\beta_i\left(e^{-h}-1\right)\\
    &=\frac{e^h-1}{h^2}B(u_{i+1}-u_i)+\frac{e^{-h}-1}{h^2}B(u_{i-1}-u_i).
\end{align*}
Since by Assumption \ref{A0} there exists $M>0$ and $a>0$ such that the following inequalities hold for all $R>M$ and $x_{i-1}>R$
\begin{equation*}
    u(x_i)-u(x_{i-1})\geq a\left(x^2_i-x^2_{i-1}\right)\geq2ax_{i-1}h+ah^2>2aRh>0.
\end{equation*}
Then by the monotonicity of $B(s)$,
\begin{equation*}
    \frac{1}{h}B(u_{i+1}-u_i)<\frac{1}{h}B(0)=\frac{1}{h},
\end{equation*}
and together with the mean value theorem one can have
\begin{equation*}
    \frac{1}{h}B(u_{i-1}-u_i)\geq\frac{1}{h}B(-2aRh)=\frac{1}{h}B(0)-2aRB'(-\xi),
\end{equation*}
where $0<\xi<2aRh$. Recall that by definition $B(0)=1$ and $B'(0)=-1/2$. Then by the continuity of $B'(s)$, there exists $\delta>0$ such that $B'(-\xi)>-1$ when $0<\xi<2aRh<\delta$. Then
\begin{equation*}
    0<\frac{1}{h}B(u_{i+1}-u_i)<\frac{1}{h}<\frac{1}{h}+2aR<\frac{1}{h}B(u_i-u_{i-1}).
\end{equation*}
It follows that
\begin{equation*}
    \frac{e^h-1}{h^2}B(u_{i+1}-u_i)+\frac{e^{-h}-1}{h^2}B(u_{i-1}-u_i)<\frac{e^h+e^{-h}-2}{h^2}-2aR\frac{(1-e^{-h})}{h}.
\end{equation*}
Since
\begin{equation*}
    \lim_{h\to0}{\frac{e^h+e^{-h}-2}{h^2}}=\lim_{h\to0}{\frac{1-e^h}{h}}=1,
\end{equation*}
then we can firstly choose a sufficiently large $R$ such that $aR\geq1$ and $R>M$. For this fixed $R$ we can find a sufficiently small $\widehat{h}_0$ to make $2aR\widehat{h}_0<\delta$, and for $0<h<\widehat{h}_0$ it holds that
\begin{equation*}
    \frac{e^h+e^{-h}-2}{h^2}-2aR\frac{(1-e^{-h})}{h}<-\widehat{\theta}<0,
\end{equation*}
where $\widehat{\theta}$ is a positive constant. Hence, we can conclude that there exist $\widehat{h}_0>0$, $R>0$ and $\widehat{\theta}>0$, such that the following inequality holds for all $0<h<\widehat{h}_0$ and $x_i>R$,
\begin{equation*}
    \mathcal{L}_hW_i\leq-\widehat{\theta}W_i.
\end{equation*}

For the negative index $i$ we can repeat the above argument and obtain a parallel result. And we claim that there exists $h_0>0$, $R>0$ and $\theta>0$ such that the following inequality holds for all $0<h<h_0$ and $|x_i|>R$,
\begin{equation*}
    \mathcal{L}_hW_i\leq-\theta W_i.
\end{equation*}
When $|x_i|\leq R$, by boundedness of continuous functions on compact set, we can always find some $b\geq0$ such that
\begin{equation*}
    \mathcal{L}_hW_i\leq-\theta W_i+bI_{B(0,R)}(x_i),
\end{equation*}
then $W$ is a Lyapunov function.

\begin{remark}
     The constants $\theta$, $b$ and $R$ are all independent of the discretization parameter $h$ for $0<h<h_0$. We make use of the properties of potential $u$, and the fact that $e^{h}+e^{-h}-2\sim\mathcal{O}(h^2)$ and $1-e^{-h}\sim\mathcal{O}(h)$ to obtain these constants.
\end{remark}

\subsection{Local Poincar\'e inequality}\label{S3.3}
In this subsection we give a proof to Lemma \ref{L2}. For the general theory of local discrete Poincaré-Sobolev inequalities, we refer to \cite[Lemma B.25]{droniou2018gradient}, which establishes these inequalities through the continuous embedding of bounded variation spaces. In this subsection we will give a more direct and brief proof for the $\mathcal{H}^1(\pi^h|_{B(0,R)})$ landscape. Assume without loss of generality that $R=Nh$ for some positive integer $N$, and denote the conditional probability as $\widehat{\pi}^h=\pi^h|_{B(0,Nh)}$ for simplicity. Since for any $i,j\in\mathbb{N}$ it holds that
\begin{equation*}
    f_i-f_j=\left\{\begin{aligned}
        &\sum^{i-1}_{k=j}{\left(f_{k+1}-f_k\right)},\quad j<i,\\
        &\sum^{j}_{k=i+1}{\left(f_{k-1}-f_k\right)},\quad j>i,
    \end{aligned}
    \right.
\end{equation*}
then
\begin{equation}\label{Eoj}
    \sum^N_{j=-N}{\widehat{\pi}_j\left(f_i-f_j\right)}=\sum^{i-1}_{j=-N}{\sum^{i-1}_{k=j}{\widehat{\pi}_j\left(f_{k+1}-f_k\right)}}+\sum^N_{j=i+1}{\sum^j_{k=i+1}{\widehat{\pi}_j\left(f_{k-1}-f_k\right)}},
\end{equation}
that is,
\begin{align*}
    f_i-\left\langle f^h,\widehat{\pi}^h\right\rangle
    &=\sum^{i-1}_{j=-N}{\sum^{i-1}_{k=j}{\widehat{\pi}_j\left(f_{k+1}-f_k\right)}}+\sum^N_{j=i+1}{\sum^j_{k=i+1}{\widehat{\pi}_j\left(f_{k-1}-f_k\right)}}\\
    &=\sum^{i-1}_{j=-N}{\sum^{i-1}_{k=j}{\frac{\widehat{\pi}_j}{i-j}(i-j)\left(f_{k+1}-f_k\right)}}+\sum^N_{j=i+1}{\sum^j_{k=i+1}{\frac{\widehat{\pi}_j}{j-i}(j-i)(f_{k-1}-f_k)}}.
\end{align*}
Therefore, \eqref{Eoj} can be regarded as an expectation with respect to the discrete density $\mu_i$ defined by
\begin{equation*}
    \mu_i(j,k):=\left\{\begin{aligned}
        &\widehat{\pi}_j/|j-i|,\quad j\leq k<i\ \mbox{ or }\ i<k\leq j,\\
        &\widehat{\pi}_i,\quad j=i,\\
        &0,\quad \mbox{others.}
    \end{aligned}
    \right.
\end{equation*}
It follows by Jensen's inequality that
\begin{align*}
    \left(f_i-\left\langle f^h,\widehat{\pi}^h\right\rangle\right)^2
    &\leq\sum^{i-1}_{j=-N}{\sum^{i-1}_{k=j}{\frac{\widehat{\pi}_j}{i-j}(i-j)^2\left(f_{k+1}-f_k\right)^2}}+\sum^N_{j=i+1}{\sum^j_{k=i+1}{\frac{\widehat{\pi}_j}{j-i}(j-i)^2(f_{k-1}-f_k)^2}}\\
    &=\sum^{i-1}_{j=-N}{\sum^{i-1}_{k=j}{\widehat{\pi}_j(i-j)\left(f_{k+1}-f_k\right)^2}}+\sum^N_{j=i+1}{\sum^j_{k=i+1}{\widehat{\pi}_j(j-i)(f_{k-1}-f_k)^2}}\\
    &=\sum^{i-1}_{k=-N}{\sum^k_{j=-N}\widehat{\pi}_j(i-j)\left(f_{k+1}-f_k\right)^2}+\sum^N_{k=i+1}{\sum^N_{j=k}{\widehat{\pi}_j}(j-i)\left(f_{k-1}-f_k\right)^2}
\end{align*}
Recall that by definition $h^2\alpha_k=B(u_{k+1}-u_k)$, $h^2\beta_k=B(u_{k-1}-u_k)$ and $R=Nh$. The following estimate relies on the key observation that the sums $\sum^k_{j=-N}{(i-j)}$ for $k<i$, and $\sum^N_{j=k}{(j-i)}$ for $k>i$ both scale as $\mathcal{O}(N^2)$. Then by introducing $h^2$ one can obtain $N^2h^2=R^2$ as well as the coefficients $\alpha_k,\beta_k$, making the constant independent of $h$ in the local discrete Poincar\'e inequality to be established. Since $\widehat{\pi}^h$ is bounded within the ball $B(0,R)$, one can have
\begin{align*}
    \left(f_i-\left\langle f^h,\widehat{\pi}^h\right\rangle\right)^2
    &\leq\sum^{i-1}_{k=-N}{\sum^k_{j=-N}{\left(\max_{-N\leq j\leq N}\widehat{\pi}_j\right)\frac{h^2(i-j)}{h^2\alpha_k}\alpha_k\left(f_{k+1}-f_k\right)^2}}\\
    &\quad+\sum^N_{k={i+1}}{\sum^N_{j=k}{\left(\max_{-N\leq j\leq N}{\widehat{\pi}_j}\right)\frac{h^2(j-i)}{h^2\beta_k}\beta_k\left(f_{k-1}-f_k\right)^2}}\\
    &\leq C_1\sum^{i-1}_{k=-N}{\alpha_k\left(f_{k+1}-f_k\right)^2}+C_2\sum^N_{k=i+1}{\beta_k\left(f_{k-1}-f_k\right)^2},
\end{align*}
where
\begin{equation*}
    \left\{\begin{aligned}
        &C_1=4R^2\left(\sup_{x\in B(0,R)}{\pi(x)}\right)\left(\inf_{x\in B(0,R)}{B\left(u(x+h)-u(x)\right)}\right)^{-1}\\
        &C_2=4R^2\left(\sup_{x\in B(0,R)}{\pi(x)}\right)\left(\inf_{x\in B(0,R)}{B\left(u(x-h)-u(x)\right)}\right)^{-1}.
    \end{aligned}
    \right.
\end{equation*}
Recall that by definition $\sum^N_{i=-N}{\widehat{\pi}_i}=1$. Then
\begin{align*}
    \left\langle\left(f^h-\left\langle f^h,\widehat{\pi}^h\right\rangle\right)^2,\widehat{\pi}^h\right\rangle
    &\leq C_1\sum^N_{i=-N}{\sum^{i-1}_{k=-N}{\widehat{\pi}_i\alpha_k\left(f_{k+1}-f_k\right)^2}}+C_2\sum^N_{i=-N}{\sum^N_{k=i+1}{\widehat{\pi}_i\beta_k\left(f_{k-1}-f_k\right)^2}}\\
    &=C_1\sum^N_{k=-N}{\sum^N_{i=k+1}{\widehat{\pi}_i\alpha_k\left(f_{k+1}-f_k\right)^2}}+C_2\sum^N_{k=-N}{\sum^{k-1}_{i=-N}{\widehat{\pi}_i\beta_k\left(f_{k-1}-f_k\right)^2}}\\
    &\leq C_1\sum^N_{k=-N}{\sum^N_{i=-N}{\widehat{\pi}_i\alpha_k\left(f_{k+1}-f_k\right)^2}}+C_2\sum^N_{k=-N}{\sum^N_{i=-N}{\widehat{\pi}_i\beta_k\left(f_{k-1}-f_k\right)^2}}\\
    &\leq\frac{1}{\kappa_R}\left\langle\Gamma^h(f^h,f^h),\widehat{\pi}^h\right\rangle,
\end{align*}
where
\begin{equation*}
    \frac{1}{\kappa_R}=2\max\{C_1,C_2\}\left(\inf_{x\in B(0,R)}{\pi(x)}\right)^{-1}.
\end{equation*}
Hence, we can conclude that for arbitrary $R>0$ there exists $\kappa_R>0$ such that the following inequality holds for all $f^h\in\mathcal{H}^1(\pi^h|_{B(0,R)})$, 
\begin{equation*}
    \left\langle(f^h)^2,\pi^h|_{B(0,R)}\right\rangle-\left\langle f^h,\pi^h|_{B(0,R)}\right\rangle^2\leq\frac{1}{\kappa_R}\left\langle\Gamma^h(f^h,f^h),\pi^h|_{B(0,R)}\right\rangle,
\end{equation*}
and the constant $\kappa_R$ is independent of $h$.

\section*{Acknowledgment}
This work was partially supported by the National Key R\&D Program of China, Project Number 2021YFA1002800.
The work of L. Li was partially supported by NSFC 12371400 and Shanghai Municipal Science and Technology Major Project 2021SHZDZX0102.
The work of J.-G.L was partially supported by NSF grant DMS-2106988.

\bibliographystyle{plain}
\bibliography{ref}

\end{document}